\documentclass{amsart}

\usepackage{amsmath,amssymb,amsfonts,enumerate,amsthm,graphicx,color}

\renewcommand{\dim}{\mbox{dim}\,}
\renewcommand{\dim}{\mbox{dim}\,}

\newcommand{\reg}{\mbox{reg}\,}

\newcommand{\pd}{\mbox{pd}\,}

\newcommand{\T}{\mathrm}

\newtheorem{thm}{Theorem}[section]
\newtheorem{cor}[thm]{Corollary}
\newtheorem{lem}[thm]{Lemma}
\newtheorem{prop}[thm]{Proposition}
\newtheorem{defn}[thm]{Definition}

\newtheorem{ques}[thm]{Question}

\numberwithin{equation}{section}

\begin{document}
\bibliographystyle{amsplain}

\title[Regularity and projective dimension of edge ideal]{Regularity and projective dimension of edge ideal of $C_5$-free vertex decomposable graphs}
\author[F. Khosh-Ahang and S. Moradi]{Fahimeh Khosh-Ahang and Somayeh Moradi$^*$}
\address{Fahimeh Khosh-Ahang, Department of Mathematics,
 Ilam University, P.O.Box 69315-516, Ilam, Iran.}
\email{khoshahang@ferdowsi.um.ac.ir}
\address{Somayeh Moradi, Department of Mathematics,
 Ilam University, P.O.Box 69315-516, Ilam, Iran and School of Mathematics, Institute
 for Research in Fundamental Sciences (IPM), P.O.Box: 19395-5746, Tehran, Iran.} \email{somayeh.moradi1@gmail.com}

\keywords{depth, edge ideal, projective dimension, regularity, vertex
decomposable\\
$*$Corresponding author}

\begin{abstract}

\noindent In this paper, we explain the regularity, projective
dimension and depth of edge ideal of some classes of graphs in
terms of invariants of graphs. We show that for a $C_5$-free vertex
decomposable graph $G$, $\T{reg}(R/I(G))= c_G$, where $c_G$ is the
maximum number of $3$-disjoint edges in $G$. Moreover for this
class of graphs we characterize $\T{pd}(R/I(G))$ and
$\T{depth}(R/I(G))$. As a corollary we describe these invariants in
forests and sequentially Cohen-Macaulay bipartite graphs.
\end{abstract}

\maketitle
\section*{Introduction}
Let $G$ be a simple graph with vertex set
$V(G)=\{x_1,\ldots,x_n\}$ and edge set $E(G)$. If $k$ is a field,
the edge ideal of $G$ in the polynomial ring $R=k[x_1,\ldots,x_n]$
is defined as $I(G)=\langle x_ix_j \ | \ \{ x_i,x_j\}\in
E(G)\rangle$. The edge ideal of a graph was first considered by
Villarreal \cite{V}. Finding connections between algebraic
properties of an edge ideal and invariants of graph, for instance
explaining the regularity, projective dimension and depth of the
ring $R/I(G)$ by some information from $G$, is of great interest.
For some classes of graphs like trees, chordal graphs  and
shellable bipartite graphs some of these invariants are studied in
\cite{HT1}, \cite{Kimura}, \cite{Moh}, \cite{Moradi}, \cite{VT} and \cite{Zheng}. Zheng in
\cite{Zheng} described regularity and projective dimension for
tree graphs. It was proved that if $G$ is a tree, then
$\reg(R/I(G))=c_G$, where $c_G$ is the maximum number of pairwise
$3$-disjoint edges in $G$. In \cite{HT1} this description of
regularity has been extended to chordal graphs. Also, Kimura in
\cite{Kimura} extended the characterization of projective
dimension in \cite{Zheng} to chordal graphs.
 Moreover, Van Tuyl in \cite{VT} proved that the equality $\reg(R/I(G))=c_G$ holds when $G$ is a sequentially
 Cohen-Macaulay bipartite graph.

 In this paper, we consider the class of $C_5$-free vertex decomposable graphs which contains the classes  of forests and sequentially
 Cohen-Macaulay bipartite graphs. For this class of graphs we show that
$\reg(R/I(G))=c_G$, which generalizes \cite [Theorem 3.3]{VT}  and
\cite [Theorem 2.18]{Zheng}. Also, we describe the projective
dimension and depth of the ring $R/I(G)$ for this class of graphs and we gain
some results that can be compared with \cite[Corollary 2.13]{Zheng} and \cite[Theorem 5.3 and
Corollary 5.6]{Kimura}.

In Section 1, we recall some definitions and theorems that we use
in the sequel. In Section 2, first we show that $\reg(R/I(G))=c_G$
for a $C_5$-free vertex decomposable graph (Theorem \ref{sh}).
Then in Corollary \ref{cor1} we deduce that this equality holds
for forests and sequentially
 Cohen-Macaulay bipartite graphs.
 The notion $d'_G$ was introduced in \cite{Kimura} and
 it was shown that for a chordal graph $G$, one has $\pd(R/I(G))=\T{bight}(I(G))=d'_G$, where $\T{bight} (I(G))$ is the maximum height among the minimal prime divisors of $I(G)$.
In Theorem \ref{bight1}, we show that for a $C_5$-free vertex
decomposable graph $G$ we also have these equalities and as some
corollaries we show this is true for chordal graphs, forests and
sequentially
 Cohen-Macaulay bipartite graphs (see Corollaries \ref{kim} and \ref{cor2}). Moreover, in Corollary
 \ref{depth} we give a description for $\T{depth}(R/I(G))$ for a
  $C_5$-free vertex decomposable graph $G$.

\section{Preliminaries}

Let $G$ be a graph with vertex set $V(G)$ and edge set $E(G)$. For
a vertex $v$ of $G$ the set of all neighborhoods of $v$ is denoted
by $N_G(v)$ or briefly $N(v)$ and we denote the set $N_G(v)\cup
\{v\}$ by $N_G[v]$ or briefly $N[v]$.  An independent set of $G$
is a subset $F\subseteq V(G)$ such that $e\not\subseteq F$, for any
$e\in E(G)$.

Vertex decomposability was introduced by Provan and Billera in
\cite{Provan+Billera} in the pure case and extended to the
non-pure case by Bj\"{o}rner and Wachs in \cite{BW} and
\cite{BW2}. We need and use the following definition of vertex
decomposable graph which is an interpretation of the definition of
vertex decomposability for the independence complex of a graph
studied in \cite{DE} and \cite{W}.

\begin{defn}\label{1.1}
A graph $G$ is recursively defined to be \textbf{vertex decomposable} if
$G$ is totally disconnected (with no edges), or if
\begin{itemize}
\item[(i)] there is a vertex $x$ in $G$ such that $G\setminus
\{x\}$ and $G\setminus N[x]$ are both vertex decomposable, and
\item[(ii)] no independent set in $G\setminus N[x]$ is a maximal
independent set in $G\setminus \{x\}$.
\end{itemize}
\end{defn}
A vertex $x$ which satisfies in the second condition is called a
\textbf{shedding vertex} of $G$.

The \textbf{Castelnuovo-Mumford regularity} (or simply regularity)
of a $\mathbb{Z}$-graded $R$-module $M$ is defined as
$$\T{reg}(M) := \max\{j-i \ | \ \beta_{i,j}(M)\neq 0\}.$$ Also, the projective dimension of $M$ is defined as
$$\T{pd}(M) := \max\{i \ | \ \beta_{i,j}(M)\neq 0 \ \text{for some}\ j\},$$
where $\beta_{i,j}(M)$ is the $(i,j)$-th graded Betti number of $M$.

Let $G$ be a graph. A subset $C\subseteq V(G)$ is called a
\textbf{vertex cover} of $G$ if it intersects all edges of $G$. A
vertex cover of $G$ is called \textit{minimal} if it has no proper
subset which is also a vertex cover of $G$. When $G$ is a graph
with $V(G)=\{x_1, \dots, x_n\}$ and $C=\{x_{i_1}, \dots,
x_{i_t}\}$ is a vertex cover of $G$, by $x^C$ we mean the monomial
$x_{i_1}\dots x_{i_t}$ in the ring of polynomials $R=k[x_1, \dots,
x_n]$. For a monomial ideal $I=\langle x_{11}\cdots
x_{1n_1},\ldots,x_{t1}\cdots x_{tn_t}\rangle$ of the polynomial
ring $R$, the \textbf{Alexander dual ideal} of $I$, denoted by
$I^{\vee}$, is defined as
$$I^{\vee}:=\langle x_{11},\ldots, x_{1n_1}\rangle\cap \cdots \cap \langle x_{t1},\ldots, x_{tn_t}\rangle.$$
One can see that, for a graph $G$,
$$I(G)^{\vee}=\langle x^C \ | \ C\ \text{is a minimal vertex cover of G} \rangle $$
and
$$I(G)=\bigcap\ (P_C \ | \ C\  \text{is a minimal vertex cover of}\  G),$$ where  $P_C=\langle x_i \ | \ x_i\in C\rangle$.
The \textbf{big height} of $I(G)$, denoted by $\T{bight}(I(G))$, is
defined as the maximum height among the minimal prime divisors of
$I(G)$, that is the maximal size of a minimal vertex cover of $G$. In fact, for a graph $G$ we have
$$\T{bight}(I(G))=\max\{|C| \ | \ C \text{\ is a minimal vertex cover of G}\}.$$

The following theorem, which was proved in \cite{T}, is one of our
main tools in the study of the regularity of the ring $R/I(G)$.

\begin{thm}(See \cite[Theorem 2.1]{T}.)\label{1.3}
Let $I$ be a square-free monomial ideal. Then
$\T{pd}(I^{\vee})=\T{reg}(R/I)$.
\end{thm}

Two edges $\{x,y\}$ and $\{w,z\}$ of $G$ are called
\textbf{$3$-disjoint} if the induced subgraph of $G$ on
$\{x,y,w,z\}$ consists of exactly two disjoint edges or
equivalently, in the complement of $G$, the induced graph on $\{x,
y,w, z\}$ is a four-cycle. A subset $A\subseteq E(G)$ is called a
\textbf{pairwise $3$-disjoint set of edges in $G$} if every pair
of distinct edges in $A$ are $3$-disjoint in $G$. Set
$$c_G=\max\{|A| \ | \ A\  \text{is a pairwise $3$-disjoint set of
edges in}\ G\}.$$

The graph $B$ with vertex set $V(B)=\{z,w_1,\dots,w_d\}$ and edge
set $E(B)=\{\{z,w_i\} \ | \ 1\leq i\leq d\}$ is called a
\textbf{bouquet}. The vertex $z$ is called the \textbf{root} of
$B$, the vertices $w_i$ \textbf{flowers} of $B$ and the edges
$\{z,w_i\}$ the \textbf{stems} of $B$. A subgraph of $G$ which is
a bouquet is called a bouquet of $G$. Let
$\mathcal{B}=\{B_1,\ldots,B_n\}$ be a set of bouquets of $G$. We
use the following notations.
$$\ \ \ F(\mathcal{B})=\{w\in V(G)\ | \ w \text{ is a flower of some bouquet in } \mathcal{B}\}$$
$$R(\mathcal{B})=\{z\in V(G) \ | \ z \text{ is a root of some bouquet in } \mathcal{B}\}$$
$$S(\mathcal{B})=\{e\in E(G) \ | \ e \text{ is a stem of some bouquet in } \mathcal{B}\}$$

Kimura in \cite{Kimura} introduced two notion of disjointness
of a set of bouquets.

\begin{defn}(See \cite[Definition 2.1]{Kimura}.)
A set of bouquets $\mathcal{B}=\{B_1,\ldots,B_n\}$ is called
\textbf{strongly disjoint}  in $G$ if
\begin{itemize}
\item[(i)] $V(B_i)\cap V(B_j)=\emptyset$ for all $i\neq j$, and
\item[(ii)] we can choose a stem $e_i$ from each bouquet $B_i\in
\mathcal{B}$ such that $\{e_1,\ldots,e_n\}$ is pairwise
$3$-disjoint in $G$.
\end{itemize}
\end{defn}

\begin{defn}(See \cite[Definition 5.1]{Kimura}.)
A set of bouquets $\mathcal{B}=\{B_1,\ldots,B_n\}$ is called
\textbf{semi-strongly disjoint}  in $G$ if
\begin{itemize}
\item[(i)] $V(B_i)\cap V(B_j)=\emptyset$ for all $i\neq j$, and
\item[(ii)] $R(\mathcal{B})$ is an independent set of
$G$.
\end{itemize}
\end{defn}

Set $$d_G:=\max\{|F(\mathcal{B})| \ | \ \mathcal{B}\
\text{is a strongly disjoint set of bouquets of } G \}$$ and
$$\ \ \ \ \ \ d'_G:=\max\{|F(\mathcal{B})| \ | \ \mathcal{B}\ \text{is a semi-strongly disjoint set of bouquets of } G \}.$$

It is easy to see that any pairwise 3-disjoint set of edges in $G$
is a strongly disjoint set of bouquets in $G$ and, any strongly
disjoint set of bouquets is semi-strongly disjoint. In this
regard, we have the inequalities
\begin{equation}\label{1.1}
c_G\leq d_G\leq d'_G.
\end{equation}
As some auxiliary tools, we need some results of \cite{HT1} and \cite{Kimura}, which present lower bounds for the regularity and projective dimension of the ring $R/I(G)$.
We end this section by recalling these results.

\begin{thm}\label{1.6}
\begin{itemize} For any graph $G$, the following hold.
\item[(i)] (See \cite[Theorem 6.5]{HT1}.)
$\T{reg}(R/I(G))\geq c_G$.
\item[(ii)] (See \cite[Theorem 3.1]{Kimura}.) $\T{pd}(R/I(G))\geq d_G$.
\end{itemize}
\end{thm}
\section{Main results}

In this section, among other things, we give some descriptions
for the regularity, projective dimension and depth  of the ring
$R/I(G)$, when $G$ is a $C_5$-free vertex decomposable graph. This
class of graphs contains some nice classes like forests and
sequentially Cohen-Macaulay bipartite graphs (see \cite[Theorem
3.2]{FV} and \cite[Theorem 2.10]{VT}). For this purpose, we use
the duality concept in Theorem \ref{1.3} and induction. In this
way, we need the following lemma.

\begin{lem}\label{exact}
Suppose that $G$ is a graph, $x\in V(G)$ and
$N_G(x)=\{y_1,\ldots,y_t\}$. Let $G'=G\setminus \{x\}$ and
$G''=G\setminus N_G[x]$. Then

\begin{itemize}
\item[(i)]  $I(G)^{\vee}=xI(G')^{\vee}+y_1\cdots
y_tI(G'')^{\vee}$;

\item[(ii)] $xy_1\cdots y_tI(G'')^{\vee}=xI(G')^{\vee}\cap
y_1\cdots y_tI(G'')^{\vee}$;

\item[(iii)] there exists a short exact sequence of $R$-modules
$$0\rightarrow xy_1\cdots
y_tI(G'')^{\vee}\rightarrow xI(G')^{\vee}\oplus y_1\cdots
y_tI(G'')^{\vee}\rightarrow I(G)^{\vee}\rightarrow 0.$$
\end{itemize}
\end{lem}

\begin{proof}
$(i)$ For any minimal vertex cover $C$ of $G$, if $x\in C$, then
clearly $C'=C\setminus \{x\}$ is a vertex cover of $G'$, so
$x^C=xx^{C'}\in xI(G')^{\vee}$. Now let $x\notin C$. Then
$\{y_1,\ldots, y_t\}\subseteq C$ and $C''=C\setminus \{y_1,\ldots,
y_t\}$ is a vertex cover of $G''$. Therefore, $x^C=y_1\cdots
y_tx^{C''}\in y_1\cdots y_tI(G'')^{\vee}$. Thus
$I(G)^{\vee}\subseteq xI(G')^{\vee}+y_1\cdots y_tI(G'')^{\vee}$.

Conversely, for any minimal vertex cover $C'$ of $G'$,
$C'\cup\{x\}$ is a vertex cover of $G$. So $xI(G')^{\vee}\subseteq
I(G)^{\vee}$. Also, if $C''$ is a minimal vertex cover of $G''$,
then $C''\cup \{y_1,\ldots, y_t\}$ is a vertex cover of $G$. Thus
$y_1\cdots y_tI(G'')^{\vee}\subseteq I(G)^{\vee}$. Therefore,
$xI(G')^{\vee}+y_1\cdots y_tI(G'')^{\vee}\subseteq I(G)^{\vee}$.
Hence, $I(G)^{\vee}=xI(G')^{\vee}+y_1\cdots y_tI(G'')^{\vee}$ as
desired.

$(ii)$ For any minimal vertex cover $C''$ of $G''$, the set
$C'=C''\cup \{y_1,\ldots, y_t\}$ is a vertex cover of $G'$. Thus,
$xy_1\cdots y_tx^{C''}=xx^{C'}\in xI(G')^{\vee}\cap y_1\cdots
y_tI(G'')^{\vee}$. Conversely, let $f\in xI(G')^{\vee}\cap
y_1\cdots y_tI(G'')^{\vee}$ be a monomial. Then $f=xf_1=y_1\cdots
y_tf_2$ for some monomials $f_1\in I(G')^{\vee}$ and $f_2\in
I(G'')^{\vee}$. So $f_2=xf_3$ for some $f_3\in I(G'')^{\vee}$.
Thus $f=xy_1\cdots y_tf_3\in xy_1\cdots y_tI(G'')^{\vee}$. The
proof is complete.

$(iii)$ By using $(i)$ and $(ii)$ in the short exact sequence
$$0\rightarrow xI(G')^{\vee}\cap y_1\cdots y_tI(G'')^{\vee}\rightarrow xI(G')^{\vee}\oplus y_1\cdots
y_tI(G'')^{\vee}\rightarrow xI(G')^{\vee}+y_1\cdots
y_tI(G'')^{\vee}\rightarrow 0$$ the result holds.
\end{proof}

Thus we can deduce:

\begin{cor}\label{ineq}
Suppose that $G$ is a graph, $x\in V(G)$ and $|N_G(x)|=t$. Let
$G'=G\setminus \{x\}$ and $G''=G\setminus N_G[x]$. Then
\begin{itemize}
\item[(i)]  $\T{pd}(I(G)^{\vee})\leq \max
\{\T{pd}(I(G')^{\vee}),\T{pd}(I(G'')^{\vee})+1\}$. \item[(ii)]
$\T{reg}(I(G)^{\vee})\leq \max
\{\T{reg}(I(G')^{\vee})+1,\T{reg}(I(G'')^{\vee})+t\}$.
\end{itemize}
\end{cor}
\begin{proof}
Considering the short exact sequence $$0\rightarrow xy_1\cdots
y_tI(G'')^{\vee}\rightarrow xI(G')^{\vee}\oplus y_1\cdots
y_tI(G'')^{\vee}\rightarrow I(G)^{\vee}\rightarrow 0$$
insures that
$$\T{pd}(I(G)^{\vee})\leq \max \{\T{pd}(xI(G')^{\vee}),\T{pd}(y_1\cdots
y_tI(G'')^{\vee}),\T{pd}(xy_1\cdots y_tI(G'')^{\vee})+1\}$$ and
$$\T{reg}(I(G)^{\vee})\leq \max \{\T{reg}(xy_1\cdots
y_tI(G'')^{\vee})-1,\T{reg}(xI(G')^{\vee}),\T{reg}( y_1\cdots
y_tI(G'')^{\vee})\}$$ (see \cite[Corollary 20.19]{E}).

On the other hand, we know that for any monomial ideal $I$ and
monomial $f$ with the property that the support of $f$ is disjoint
from the support of any generators of $I$, we have
$\T{pd}(fI)=\T{pd}(I)$ and $\T{reg}(fI)=\T{reg}(I)+\deg(f)$. This
fact completes the proof.
\end{proof}

The following lemma is needed in the sequel frequently.
\begin{lem}\label{shedding}
Assume that $G$ is a $C_5$-free graph and $x$
is a shedding vertex of $G$. Then there is a vertex $y$ of $N(x)$
such that $N[y]\subseteq N[x]$.
\end{lem}
\begin{proof}
Let $x$ be a shedding vertex of a $C_5$-free graph $G$ and $N(x)=\{y_1, \dots, y_t\}$. Suppose, in contrary,
that for each $1\leq i \leq t$, there exists a vertex $w_i$ in
$N(y_i)\cap (V(G)\setminus N[x])$. Now, if $w_i$ is adjacent to $w_j$
for some $1\leq i, j\leq t$ with $i\neq j$, then
$x-y_i-w_i-w_j-y_j-x$ is a $C_5$-subgraph of $G$ and so it is a
contradiction. Hence, $\{w_1, \dots, w_t\}$ is an independent set
in $G\setminus N[x]$. Now, if $F$ is a maximal independent set in $G\setminus N[x]$ containing $\{w_1, \dots, w_t\}$, it is also a a maximal independent set in
$G\setminus \{x\}$. This contradicts with our assumption that $x$ is a shedding vertex and so completes our proof.
\end{proof}

Now we are ready to describe the regularity of $R/I(G)$ for a $C_5$-free
vertex decomposable graph $G$.

\begin{thm}\label{sh}
Let $G$ be a $C_5$-free vertex decomposable graph. Then
$$\T{reg}(R/I(G))=c_G.$$
\end{thm}
\begin{proof}
In view of Theorems \ref{1.6}(i) and \ref{1.3} it is enough to show that $\T{pd}(I(G)^{\vee})\leq c_G$.
We proceed by induction on $|V(G)|$. For $|V(G)|=2$, $G$ is totally disconnected or a single edge. Hence, $I(G)^\vee=0$ or $I(G)^\vee=\langle x,y \rangle$. Therefore, $\T{pd}(I(G)^\vee )=0\leq 0=c_G$ or $\T{pd}(I(G)^\vee )=1\leq 1=c_G$. Suppose that $G$ is a $C_5$-free vertex decomposable graph with $|V(G)|>2$ and the result holds for each $C_5$-free vertex decomposable graph $H$ with smaller values of $|V(H)|$. Since $G$ is vertex decomposable, there exists a shedding vertex $x\in V(G)$ such that $G'=G\setminus \{x\}$ and $G''=G\setminus N_G[x]$ are vertex decomposable. Let $N_G(x)=\{y_1, \dots, y_t\}$. By Corollary \ref{ineq}(i), we have
 $$\T{pd}(I(G)^{\vee})\leq\max\{\T{pd}(I(G')^{\vee}),\T{pd}(I(G'')^{\vee})+1\}.$$
 Clearly $G'$ and $G''$ are $C_5$-free and vertex decomposable. So, by induction
hypothesis we have $\T{pd}(I(G')^{\vee})\leq c_{G'}$ and
$\T{pd}(I(G'')^{\vee})\leq c_{G''}$. Hence,
$$\T{pd}(I(G)^{\vee})\leq\max\{c_{G'},c_{G''}+1\}.$$
Now, since $c_{G'}\leq c_G$, it is enough to show that $c_{G''}+1\leq c_G$.

By Lemma \ref{shedding} there is
 a vertex $y$ such that $N[y]\subseteq N[x]$, thus we can add the edge $\{x,y\}$ to any set of pairwise $3$-disjoint edges in $G''$ and get a
pairwise $3$-disjoint set of edges in $G$, which proves the inequality $c_{G''}+1\leq c_G$.
\end{proof}

As a corollary we can recover  results of Zheng \cite{Zheng} and Van Tuyl \cite{VT} as follows.
\begin{cor}\label{cor1}
The following hold.
\begin{itemize}
\item[(i)] \cite [Theorem 3.3]{VT}  Let $G$ be a sequentially Cohen-Macaulay bipartite graph. Then $\T{reg}(R/I(G))=c_G$.
\item[(ii)] \cite [Theorem 2.18]{Zheng} Let $G$ be a forest. Then $\T{reg}(R/I(G))=
c_G$.
\end{itemize}

\end{cor}
 \begin{proof}
 $(i)$ By \cite [Theorem 2.10]{VT}, $G$ is vertex decomposable. Since a bipartite graph is $C_5$-free, Theorem \ref{sh} yields the result.

$(ii)$ Since any forest is a chordal graph, it is vertex decomposable by \cite[Corollary 7]{W}. Clearly $G$ is
also $C_5$-free. So we can apply Theorem \ref{sh} to get the result.
\end{proof}

The following result presents an upper bound for the projective dimension of the ring $R/I(G)$. As we shall see later, this is a technical tool for characterizing the projective dimension of the ring $R/I(G)$.

\begin{prop}\label{pd}
Let $G$ be a $C_5$-free vertex decomposable graph. Then
$$\T{pd}(R/I(G))\leq d'_G.$$
\end{prop}

\begin{proof}
From the equalities $I(G)=(I(G)^{\vee})^{ \vee}$, $\T{pd}(R/I(G))=\T{pd}(I(G))+1$ and $\T{reg}(I(G))=\T{reg}(R/I(G))+1$ and
Theorem \ref{1.3}, one can see that $\T{pd}(R/I(G))=\T{reg}(I(G)^\vee)$. So, it is enough to show that $\T{reg}(I(G)^{\vee})\leq
d'_G$. We prove the
assertion by induction on $|V(G)|$ . For $|V(G)|=2$, $G$ is totally disconnected or a single edge. Hence, $\T{reg}(I(G)^\vee)= 0\leq 0=d'_G$ or $\T{reg}(I(G)^\vee)=1\leq 1=d'_G$. Now, suppose inductively that $G$ is a $C_5$-free vertex decomposable graph with $|V(G)|>2$ and the result holds for smaller values of $|V(G)|$. Assume that $x\in V(G)$ is a shedding vertex of $G$ such that $G'=G\setminus \{x\}$ and $G''=G\setminus N_G[x]$ are vertex decomposable and $N_G(x)=\{y_1,\dots,y_t\}$. In view of Corollary \ref{ineq}(ii), we have
$$\T{reg}(I(G)^{\vee})\leq \max \{\T{reg}(I(G')^{\vee})+1,\T{reg}(I(G'')^{\vee})+t\}.$$
By induction hypothesis $\T{reg}(I(G')^{\vee})\leq d'_{G'}$ and
$\T{reg}(I(G'')^{\vee})\leq d'_{G''}$.
Thus
\begin{equation}\label{1}
\T{reg}(I(G)^{\vee})\leq \max \{d'_{G'}+1,d'_{G''}+t\}.
\end{equation}
Now, let $\mathcal{B}=\{B_1,\ldots,B_n\}$ be a semi-strongly disjoint set
of bouquets in $G''$ with $d'_{G''}=|F(\mathcal{B})|$. Then by adding the bouquet with root
$x$ and flowers $\{y_1,\ldots,y_t\}$ to $\mathcal{B}$, we have a semi-strongly disjoint set
of bouquets $\mathcal{B'}$ in $G$ with $|F(\mathcal{B'})|=d'_{G''}+t$. Therefore
\begin{equation}\label{2}
d'_{G''}+t\leq d'_G.
\end{equation}
Now let $\mathcal{B}=\{B_1,\ldots,B_n\}$ be a semi-strongly disjoint set
of bouquets in $G'$ with $d'_{G'}=|F(\mathcal{B})|$. We consider the following two cases.

\textbf{Case I.} If $y_i\in R(\mathcal{B})$ for some $1\leq i\leq t$, then by adding the stem $\{x,y_i\}$
to the bouquet with root $y_i$, $G$ has a semi-strongly disjoint set
of bouquets $\mathcal{B'}$ with $|F(\mathcal{B'})|=d'_{G'}+1$.

\textbf{Case II.} If $R(\mathcal{B})\cap \{y_1,\ldots,y_t\}=\emptyset$, then by Lemma \ref{shedding},
there exists $1\leq i\leq t$, such that $N[y_i]\subseteq N[x]$.
So $y_i\notin R(\mathcal{B})\cup F(\mathcal{B})$. Thus, adding the bouquet with a single stem $\{x,y_i\}$ to $\mathcal{B}$, induces a semi-strongly disjoint set
of bouquets $\mathcal{B'}$ of $G$ with $|F(\mathcal{B'})|=d'_{G'}+1$.

Therefore, in each case we have
\begin{equation}\label{3}
d'_{G'}+1\leq d'_G.
\end{equation}
Now (\ref{1}), (\ref{2}) and (\ref{3}) implies the result.
\end{proof}

Dao and Schweig in \cite{Dao+Schweig} introduce a new graph domination parameter called edgewise domination. Let $F\subseteq E(G)$. We say that $F$ is edgewise dominant if any $v\in V(G)$ is adjacent to an endpoint of some edge $e\in F$. They also define
$$\epsilon (G):=\min\{|F| \ | \ F\subseteq E(G) \ \text{is edgewise dominant}\}.$$
Moreover, recall that $S
\subseteq V(G)$ is called a dominating set of $G$ if each vertex in $V(G)\setminus S$ is adjacent to some vertex in $S$. Also,
$$\gamma (G)=\min\{|A| \ | \ A \ \text{is a dominating set of} \ G\}$$
is another graph domination parameter.

The following proposition declares when a semi-strongly disjoint set of bouquets of a graph
corresponds to a minimal vertex cover. The argument is the same as in \cite[Corollary 5.6]{Kimura}.

\begin{prop}\label{cover}
Let $G$ be a graph and $\mathcal{B}=\{B_1,\ldots,B_n\}$ be a semi-strongly disjoint set of bouquets in $G$
with $|F(\mathcal{B})|=d'_G$. Then
\begin{itemize}
\item[(i)] $F(\mathcal{B})$ is a minimal vertex cover of $G$.
\item[(ii)] If moreover, $G$ has no isolated vertex, we have $F(\mathcal{B})$ is a dominating set of $G$ and $S(\mathcal{B})$ is edgewise dominant in $G$.
\end{itemize}
\end{prop}
\begin{proof}
(i) First we show that the set $F(\mathcal{B})$
is a vertex cover of $G$. Assume, in contrary, that $\{x,y\}$ is an edge which is not covered
by $F(\mathcal{B})$. Then $x,y\notin F(\mathcal{B})$.
Moreover $x$ and $y$ are not adjacent to any vertex in $R(\mathcal{B})$, otherwise
if $\{x,z\}\in E(G)$ for some $z\in R(\mathcal{B})$, then by adding the stem $\{x,z\}$ to
the bouquet with root $z$, we have a semi-strongly disjoint set
of bouquets $\mathcal{B}'$ with $|F(\mathcal{B}')|=d'_G+1$, which is a contradiction. Moreover $x,y\notin R(\mathcal{B})$, otherwise if  $x$ (respectively
$y$) is a root, then $y$ (respectively $x$) is adjacent to a vertex in $R(\mathcal{B})$, which is not possible by the above argument.
Therefore, if we add the bouquet with a single stem $\{x,y\}$ to the set $\mathcal{B}$,
we have a semi-strongly disjoint set
of bouquets $\mathcal{B}'$ with $|F(\mathcal{B}')|=d'_G+1$, which is again a contradiction.
So $F(\mathcal{B})$
is a vertex cover of $G$ as desired. Moreover, it is a minimal one, since removing any
flower makes the corresponding stem uncovered.

(ii) Let $v\in V(G)$. Since $G$ has no isolated vertex, there exists an edge $e$ containing $v$.
Since $F(\mathcal{B})$ is a vertex cover of $G$, $F(\mathcal{B})\cap e\neq \emptyset$.
This means that there exists an stem  $e'\in S(\mathcal{B})$ such that  $v$ is adjacent to one endpoint of $e'$. Therefore, $S(\mathcal{B})$ is edgewise dominant. Since $G$ has no isolated vertex, any vertex cover is a dominating set. Hence, (i) insures that $F(\mathcal{B})$ is a dominating set of $G$.
\end{proof}

The following corollary provides a chain of inequalities between some algebraic invariants of the edge ideal $I(G)$ and some invariants of $G$.
\begin{cor}\label{bight}
For a graph $G$ we have
$$c_G\leq d_G\leq d'_G\leq \T{bight}(I(G))\leq \T{pd}(R/I(G))\leq |V(G)|-\epsilon(G).$$
If moreover $G$ has no isolated vertex, we have
$$\max\{\epsilon(G), \gamma (G)\}\leq d'_G.$$
\end{cor}
\begin{proof}
The first assertion can be gained from the inequality (\ref{1.1}), Proposition \ref{cover} (i), \cite [Theorem 3.31]{Vnew}
and \cite[Theorem 4.3]{Dao+Schweig}.
The second one is an immediate consequence of Proposition \ref{cover} (ii).
\end{proof}

Now, we are ready to bring another main result of this note. This shows that the upper bound determined in Proposition \ref{pd} is tight and also, it is equal to big height of the edge ideal.
\begin{thm}\label{bight1}
Let $G$ be a $C_5$-free vertex decomposable graph. Then
$$\T{pd}(R/I(G))=\T{bight}(I(G))=d'_G.$$
\end{thm}

\begin{proof}
By Proposition \ref{pd} and Corollary \ref{bight} the result holds.
\end{proof}

In \cite{Kimura}, it was proved that for a chordal graph $G$, $\T{pd}(R/I(G))=d_G$ and later with another argument it was shown that for a chordal graph $G$, we moreover have $\T{pd}(R/I(G))=d_G=d'_G$ (see \cite[Theorems 4.1 and 5.3]{Kimura}. Now, the next corollary shows that Theorem 5.3 and Corollary 5.6 in \cite{Kimura} can be directly gained by
Corollary \ref{bight} and \cite[Theorem 4.1]{Kimura}.

\begin{cor}(See \cite[Theorem 5.3 and Corollary 5.6]{Kimura}.)\label{kim}
Let $G$ be a chordal graph. Then
$$\T{pd}(R/I(G))=\T{bight}(I(G))=d_G=d'_G.$$
\end{cor}

Looking at many examples of $C_5$-free vertex decomposable graphs (which were not necessarily chordal),
we observed that $d_G=d'_G$. So the following question
comes to mind:

\begin{ques}
Does the equality $d_G=d'_G$ hold for a $C_5$-free vertex decomposable graph?
\end{ques}

In the next, we are interested in characterizing the depth of the ring $R/I(G)$, when $G$ is a $C_5$-free vertex decomposable graph. Recall that a graph $G$ is called unmixed if all maximal independent sets in $G$ have the same cardinality.
\begin{cor}\label{depth}
Let $G$ be a $C_5$-free vertex decomposable graph. Then
$$\T{depth}(R/I(G))=\min \{|F| \ | \ F\subseteq V(G) \ \text{is a maximal independent set in} \ G\}.$$
Moreover,  $R/I(G)$ is Cohen-Macaulay if and only if  $G$ is unmixed.
\end{cor}
\begin{proof}
By applying Auslander-Buchsbaum formula for $R/I(G)$, we have
$$\T{pd}(R/I(G))+\T{depth}(R/I(G))=n,$$
where $n=|V(G)|$. So, by Theorem \ref{bight1},
$$\T{depth}(R/I(G))=n-\T{bight}(I(G)).$$
In view of the definition of the big height of $I(G)$, there exists a minimal vertex cover $C$
of $G$ with $\T{bight}(I(G))=|C|$. Since every minimal vertex cover $C'$ of $G$ corresponds
to the maximal independent set $F=V(G)\setminus C'$ and $C$ has the maximal cardinality
among minimal vertex covers, the first assertion holds.
Now, since $$\T{depth}(R/I(G))=\min \{|F| \ | \ F \ \text{is a maximal independent set in} \ G\}$$
and by \cite[Corollary 5.3.11]{VIL} we have
$$\dim(R/I(G))=\max \{|F| \ | \ F \ \text{is an independent set in}\ G\},$$

Cohen-Macaulayness is equivalent to say that all maximal independent sets in $G$ have the same cardinality. This yields the result.
\end{proof}

By considering the fact that forest graphs and sequentially Cohen-Macaulay bipartite graphs
are $C_5$-free vertex decomposable and applying Theorem \ref{bight1} and Corollary \ref{depth} one has:

\begin{cor}\label{cor2}
Let $G$ be a forest or a sequentially Cohen-Macaulay bipartite graph. Then
\begin{itemize}
\item[(i)] $\T{pd}(R/I(G))=d'(G)=\T{bight}(I(G))$.
\item[(ii)] $\T{depth}(R/I(G))=\min \{|F| \ | \ F\subseteq V(G) \ \text{is a maximal independent set in} \ G\}.$
\item[(iii)] the ring $R/I(G)$ is Cohen-Macaulay if and only if $G$ is unmixed.
\end{itemize}
\end{cor}

\textbf{Acknowledgments:}
We would like to thank the referee for his fruitful comments and suggestions. The research of the second author was in part supported by a grant from IPM (No. 900130065).

\providecommand{\bysame}{\leavevmode\hbox
to3em{\hrulefill}\thinspace}

\end{document}